\documentclass[11pt,reqno]{amsart}
\usepackage{amsthm, amsfonts, amssymb, color}
 \usepackage{mathrsfs}
\usepackage{enumerate}
\usepackage{amsmath}
 \usepackage{amstext, amsxtra}
  \usepackage{txfonts}
 \usepackage[colorlinks, linkcolor=black, citecolor=blue, pagebackref, hypertexnames=false]{hyperref}

 \allowdisplaybreaks
\newtheorem{theorem}{Theorem}[section]
\newtheorem{proposition}[theorem]{Proposition}

\newtheorem{lemma}[theorem]{Lemma}

\newtheorem{remark}[theorem]{Remark}

\newcommand\R{\mathbb{R}}

\newcommand{\la}{\lambda}

\newcommand{\supp}{{\rm supp}{\hspace{.05cm}}}

\numberwithin{equation}{section}
\theoremstyle{definition}

\title
 [Schr\"{o}dinger-Poisson systems involving critical nonlocal term]
  {Multiple positive solutions for Schr\"{o}dinger-Poisson systems involving critical nonlocal term}
\author{Liejun Shen\ and \ Xiaohua Yao}

\address{ Liejun Shen,  Hubei Key Laboratory of Mathematical Sciences and School of Mathematics and Statistics,
 Central China Normal University, Wuhan, 430079, P. R. China }
\email{liejunshen@sina.com}
\address{Xiaohua Yao, Hubei Key Laboratory of Mathematical Sciences and School of Mathematics and Statistics,
 Central China Normal University, Wuhan, 430079, P.R. China}
\email{yaoxiaohua@mail.ccnu.edu.cn }
\date{\today}

\begin{document}

\maketitle
\begin{abstract} The present study is concerned with the following Schr\"{o}dinger-Poisson system involving critical nonlocal term
$$
\left\{%
\begin{array}{ll}
    -\Delta u+u-K(x)\phi |u|^3u=\la f(x)|u|^{q-2}u, & x\in\R^3, \\
    -\Delta \phi=K(x)|u|^5, &  x\in\R^3,\\
\end{array}%
\right.
$$
where $1<q<2$ and $\lambda>0$ is a parameter. Under suitable assumptions on $K(x)$ and $f(x)$, there exists $\lambda_0=\lambda_0(q,S,f,K)>0$ such that for any $\lambda\in(0,\lambda_0)$, the above Schr\"{o}dinger-Poisson system possesses at least two positive solutions by standard variational method, where a positive least energy solution will also be obtained.
\end{abstract}

%\tableofcontents

\section{Introduction and Main Results}
Due to the real physical meaning, the following Schr\"{o}dinger-Poisson system
\begin{equation}\label{introduction1}
 \left\{%
\begin{array}{ll}
    -\Delta u+V(x)u+\phi u=f(x,u), & x\in\R^3, \\
    -\Delta \phi=u^2, &  x\in\R^3,\\
\end{array}%
\right.
\end{equation}
has been studied extensively by many scholars in the last several decades. The system like \eqref{introduction1} firstly introduced by Benci and Fortunato \cite{Benci1} was used to
describe solitary waves for nonlinear Sch\"{o}rdinger type equations and look for the existence of standing waves interacting with an unknown electrostatic field. We refer the readers to \cite{Benci1,Benci2,P. L. Lions3,Markowich} and the references therein to get a more physical background of the system \eqref{introduction1}.

In recent years, by classical variational
methods, there are many interesting works about the existence and non-existence of positive solutions, positive
ground states, multiple solutions, sign-changing solutions and semiclassical states
to the system \eqref{introduction1} with different assumptions on the potential $V(x)$ and the nonlinearity $f(x,u)$ were established. If $V(x)\equiv 1$ and $f(x,u)=|u|^{p-1}u$, T. d'Aprile and D. Mugnai \cite{dAprile2} showed that the system \eqref{introduction1} has no nontrivial solutions when $p\leq 1$ or $p\geq 5$. For the case $4\leq p<6$, the existence of radial and non-radial solutions was studied in \cite{Coclite,dAprile1,dAprile3}. D. Ruiz \cite{Ruiz} proved the existence and nonexistence of nontrivial solutions when $1<p<5$. When $V(x)\equiv0$ and $f(x,u)=g(u)$,
A. Azzollini, P. d'Avenia and A. Pomponio \cite{Azzollini1} investigated the existence of nontrivial radial solutions when $\mu\in(0,\mu_0)$ for the following system
\begin{equation}\label{introduction4}
   \left\{%
\begin{array}{ll}
    -\Delta u+ \mu\phi u=g(u), & x\in\R^3, \\
    -\Delta \phi=\mu u^2, &  x\in\R^3,\\
\end{array}%
\right.
\end{equation}
under the conditions $g\in C(\R)$ and
\[
 (H_1)\ \  -\infty<\mathop{\lim\inf}_{t\to0}\frac{g(t)}{t}\leq\mathop{\lim\sup}_{t\to0}\frac{g(t)}{t}=-m<0;\ \ \ \ \ \ \ \ \ \ \ \ \ \ \ \   \ \ \ \ \ \ \ \ \ \ \ \   \ \ \ \ \ \ \ \ \ \ \ \ \ \ \ \   \ \ \ \ \ \ \ \ \ \ \ \ \ \ \ \
\]
\[
 (H_2)\ \  -\infty\leq\mathop{\lim\sup}_{t\to\infty}\frac{g(t)}{t^{{5}}}\leq0;\ \ \ \ \ \ \ \ \ \ \ \ \ \ \ \   \ \ \ \ \ \ \ \ \ \ \ \   \ \ \ \  \ \ \ \ \ \ \ \ \ \ \ \   \ \ \  \ \ \ \ \ \ \ \ \ \ \ \   \ \ \ \ \ \ \  \ \ \ \ \ \ \ \   \ \ \ \ \ \ \ \ \ \ \ \   \ \ \ \  \ \ \  \ \ \ \ \ \ \ \ \ \ \ \ \ \ \  \ \ \ \ \ \ \ \ \ \ \ \
\]
\[
 (H_3)\ \  \text{there exists}\ \ \xi>0\ \  \text{such that}\ \ \int_0^\xi g(t)dt>0.\ \ \ \ \ \ \ \ \ \ \ \ \ \ \ \   \ \ \ \ \ \ \ \ \ \ \ \   \ \ \ \ \ \  \ \ \ \ \ \ \ \   \ \ \ \ \ \ \ \ \ \ \ \   \ \ \ \  \ \ \ \ \ \ \ \
\]
We mention here that the hypotheses $(H_1)-(H_3)$ are the so-called Berestycki-Lions conditions, which were introduced
in H. Berestycki and P. L. Lions \cite{Berestycki} for the derivation of the ground state solution of \eqref{introduction4}. If $V(x)\not\equiv constant$ and $f(x,u)=|u|^{p-1}u+\mu |u|^{4}u$ with $2<p<5$, the existence of positive ground state was obtained by  Z. Liu and S. Guo \cite{Liu0}. By using superposition principle established by
N. Ackermann \cite{Ackermann}, the system \eqref{introduction1} with a periodic potential was studied by  J. Sun and S. Ma \cite{Sun}, where the existence of infinitely many geometrically distinct solutions was proved. For other related and important results, we refer the readers to \cite{Alves2,He,Huang,Jiang,Zhao} and their
references.

However, the results for the following general Schr\"{o}dinger-Poisson system
\begin{equation}\label{mmmmmm}
  \left\{%
\begin{array}{ll}
    -\Delta u+u+p \phi g(u)=f(x,u), & x\in\R^3, \\
    -\Delta \phi=2pG(u), &  x\in\R^3\\
\end{array}%
\right.
\end{equation}
are not so fruitful as the case $g(u)=u$ and $p\in\R$, where $|g(t)|\leq C(|t|+|t|^s)$ with $s\in (1,4)$, please see \cite{Azzollini3,Li2} for example. When $s=4$ in \eqref{mmmmmm},
A. Azzollini and P. d'Avenia \cite{Azzollini2} firstly studied the following Schr\"{o}dinger-Poisson system with critical nonlocal term
\begin{equation}\label{aaaaaaaa}
 \left\{%
\begin{array}{ll}
    -\Delta u=\mu u+p\phi|u|^3u, & x\in B_R, \\
    -\Delta \phi=p|u|^5, &  x\in B_R,\\
    u=\phi=0,& \text{on} \ \ \partial B_R.
\end{array}%
\right.
\end{equation}
 Note that although the second equation can be solved by a Green's function, the term $p|u|^5$ will result in a nonlocal critically growing nonlinearity in \eqref{aaaaaaaa}. After it, by using a monotonic trick introduced by L. Jeanjean \cite{Jeanjean}, F. Li, Y. Li and J. Shi \cite{Li} specially proved
\[
  \left\{%
\begin{array}{ll}
    -\Delta u+bu-\phi |u|^3u=f(u), & x\in\R^3, \\
    -\Delta \phi=|u|^5, &  x\in\R^3,\\
\end{array}%
\right.
\]
possesses at least one positive radially symmetric solution when $b>0$ is a constant.

In their celebrated paper, A. Ambrosetti, H. Br\'{e}zis, G. Cerami \cite{Ambrosetti} studied the following semilinear elliptic equation with concave-convex nonlinearities:
\begin{equation}\label{c}
 \left\{%
\begin{array}{ll}
    - \Delta u = |u|^{p-2}u+\mu |u|^{q-2}u, &\text{in}\ \ \Omega, \\
    u=0, &  \text{on}\ \  \partial\Omega,\\
\end{array}%
\right.
\end{equation}
where $\Omega$ is a bounded domain in $\R^N$ with $\mu>0$ and $1<q<2<p\leq 2^*=\frac{2N}{N-2}$. By variational method, they have obtained the existence and multiplicity of positive solutions to the problem \eqref{c}. Subsequently, an increasing number of researchers have
paid attention to semilinear elliptic equations with critical exponent and concave-convex nonlinearities,
for example, see \cite{Barstch,Bouchekif,Chen,Hsu,Huang0,Lei,Nyamoradi,Wu} and the references therein.

To the best of our knowledge, the Schrodinger-Poisson system with critical nonlocal term was only studied in \cite{Azzollini2,Li,Liu}. Meanwhile there are very few papers on existence of multiple results for Schrodinger-Poisson system with concave-convex nonlinearities. Inspired by the works mentioned above, this paper will fill the gap and
prove the existence of multiple positive solutions for the following Schrodinger-Poisson system:
\begin{equation}\label{mainequation}
  \left\{%
\begin{array}{ll}
    -\Delta u+u-K(x)\phi |u|^3u=\la f(x)|u|^{q-2}u, & x\in\R^3, \\
    -\Delta \phi=K(x)|u|^5, &  x\in\R^3,\\
\end{array}%
\right.
\end{equation}
where $1<q<2$ and $\la>0$ ia a parameter. The assumptions on $K(x)$ and $f(x)$ are as follows:
\begin{enumerate}[$(K)$]
  \item \emph{$K(x)\in C(\R^3,\R^+)$, there exist some constants $C>0$, $\delta>0$ and $\beta\in[1,3)$ such that
$$
|K(x)-K(x_0)|\leq C|x-x_0|^\beta\ \ \text{if}\ \ |x-x_0|<\delta,
$$
and $x_0\in\R^3$ satisfies $K(x_0)=\max_{x\in\R3}K(x):=|K|_\infty<+\infty$.}
\end{enumerate}

\begin{enumerate}[$(F)$]
\item \emph{$f\in L^{\frac{2}{2-q}}(\R^3)$ with $f(x)\geq 0$ and $f(x)\not\equiv 0$.}
\end{enumerate}

Now we state our main result:

\begin{theorem}\label{maintheorem}
Assume $(K)$ and $(F)$, for any $1<q<2$ there exists $\la_0=\la_0(q,S,f,K)>0$ such that the system \eqref{mainequation} admits at least two positive solutions when $\la\in(0,\la_0)$. In addition, a positive least energy solution can also be established.
\end{theorem}

\begin{remark}
\
\begin{enumerate}[$(1)$]
  \item There are a lot of functions to meet the assumption $(K)$ such as $K\equiv1$. Without doubt that $K(x_0)>0$ is necessary and otherwise $K(x)\equiv 0$ implies that the critical term disappears and the system \eqref{mainequation} degenerates to a semilinear Schr\"{o}dinger equation. This kind of assumption $(K)$ was firstly introduced by F. Gazzola and M. Lazzarino \cite{Gazzola} to consider a semilinear Schr\"{o}dinger equation.
  \end{enumerate}
  \begin{enumerate}[$(2)$]
  \item  The assumption of non-negativity for $f(x)$ is not essential to the analysis of Theorem \ref{maintheorem}. In a word, the method used in Theorem \ref{maintheorem} can deal with the case when $f(x)$ is sign-changing.
  \end{enumerate}
\end{remark}

\begin{remark}
It is worth to point out here that we not only show the existence of $\la_0$ obtained in Theorem \ref{maintheorem}, but also give the concrete expression:
\[
\lambda_0=\la_0(q,S,f,K):=\frac{4q}{(10-q)|f|_{\frac{2}{2-q}}}\bigg(\frac{5S^6(2-q)}{|K|_\infty^2(10-q)}\bigg)^{\frac{2-q}{8}}>0,
\]
where $S>0$ is the best Sobolev constant \big(see \eqref{Sobolev}\big).
\end{remark}

The nonlocal critical term in \eqref{mainequation} makes the problem complicated because of the lack of compactness imbedding of $H^1(\R^3)$ into $L^r(\R^3)$ for $r\in[2,6]$. Moreover we do not assume that the functions $K(x)$ and $f(x)$ are radial symmetric, so it is impossible to work in the radial symmetric space. To overcome it, the assumption on $f(x)$ plays an vital role. However if we replace $\R^3$ by a bounded domain $\Omega$, the above difficult disappears. Of course, the assumption on $K(x)$ can never make a contribution to recovering the compactness. What we want to emphasize is that either $K(x)\equiv1$ or $K(x)$ satisfies $(K)$ in our problem, the proof does not have an essential difficult, but this difference seems to cause some special obstacles in \cite{Li} with this case. Meanwhile, by means of a totally same idea but some simpler calculations employed in Theorem \ref{maintheorem}, one immediately has the following result which will not be proved in detail.

\begin{theorem}
Assume $\Omega\subset \R^3$ is a bounded domain with a smooth boundary $\partial\Omega$, then
for any $1<q<2$ there exists $\widetilde{\la}_0=\widetilde{\la}_0(q,S)>0$ such that the conclusions obtained in Theorem \ref{maintheorem} still holds for
the following system
 \[
  \left\{%
\begin{array}{ll}
    -\Delta u-\phi |u|^3u=\la |u|^{q-2}u, & x\in\Omega, \\
    -\Delta \phi= |u|^5, &  x\in\Omega,\\
    u= \phi=0, &   \text{on}\ \   \partial\Omega,\\
\end{array}%
\right.
\]
when $ {\la}\in(0,\widetilde{\la}_0)$.
\end{theorem}

The outline of this paper is as follows. In Section 2, we present some preliminary results for Theorem \ref{maintheorem}. In Section 3, we will prove Theorem \ref{maintheorem}.
\\\\
\textbf{Notations.} Throughout this paper we shall denote by $C$ and $C_i$ ($i=1, 2,\cdots$) for various positive constants whose exact value may change from lines to lines but are not essential to the analysis of problem. We use $``\to"$ and $``\rightharpoonup"$ to denote the strong and weak convergence in the related function space, respectively. For any $\rho>0$ and any $x\in \R^3$, $B_\rho(x)$ denotes the ball of radius $\rho$ centered at $x$, that is, $B_\rho(x):=\{y\in \R^3:|y-x|<\rho\}$.

Let $(X,\|\cdot\|)$ be a Banach space with its dual space $(X^{*},\|\cdot\|_{*})$, and $\Psi$ be its functional on $X$. The Palais-Smale sequence at level $d\in\R$ ($(PS)_d$ sequence in short) corresponding to $\Psi$ satisfies that $\Psi(x_n)\to d$ and $\Psi^{\prime}(x_n)\to 0$ as $n\to\infty$, where $\{x_n\}\subset X$.

\section{Some Preliminaries}
In this section, we will give some lemmas which are useful for the main results. To solve the system \eqref{mainequation}, we introduce some function spaces. Throughout the paper, we consider the Hilbert space $H^1(\R^3)$ with the usual inner product
$$
(u,v)=\int_{\R^3}\nabla u\nabla v+uvdx,\ \ \ \  \ \   \forall u,v\in H^1(\R^3)
$$
and the norm
$$
\|u\|=\bigg(\int_{\R^3}|\nabla u|^2+ u^2dx\bigg)^{\frac{1}{2}}.
$$
$L^r(\R^3)$ ($1\leq r<+\infty$) is the Lebesgue space, $|\cdot|_r$ means its usual $L^r$-norm and
the space
\[
D^{1,2}(\R^3)=\big\{u\in L^6(\R^3): \nabla u \in L^2(\R^3)\big\}
\]
equips with its usual norm and inner product
$$
\|u\|_{D^{1,2}(\R^3)}=\bigg(\int_{\R^3}|\nabla u|^2dx\bigg)^{\frac{1}{2}}\ \ \text{and}\ \  (u,v)_{D^{1,2}(\R^3)}=\int_{\R^3}\nabla u\nabla vdx,
$$
respectively. The positive constant $S$ denotes the best Sobolev constant:
\begin{equation}\label{Sobolev}
  S:=\inf_{u\in D^{1,2}(\R^3)\setminus \{0\}}\frac{\int_{\R^3}|\nabla u|^2dx}{\big(\int_{\R^3}|u|^6dx\big)^{\frac{1}{3}}}.
\end{equation}

In the following, one can use the Lax-Milgram theorem, for every $u\in H^1(\R^3)$, there exists a unique $\phi_{u}\in D^{1,2}(\R^3)$ such that
\begin{equation}\label{Lax}
  -\Delta \phi=K(x)|u|^5
\end{equation}
 and $\phi_{u}$ can be written as
\begin{equation}\label{Riesz}
  \phi_{u}(x)=\frac{1}{4\pi}\int_{\R^3}\frac{K(y)|u(y)|^5}{|x-y|}dy.
\end{equation}
Substituting \eqref{Riesz} into \eqref{mainequation}, we get a single elliptic equation with a nonlocal term:
\[
-\Delta u+u-K(x)\phi_u|u|^3u=\lambda f(x)|u|^{q-2}u
\]
whose corresponding functional $J:H^1(\R^3)\to\R$ is defined by
\begin{equation}\label{functional}
  J(u)=\frac{1}{2}\|u\|^2-\frac{1}{10}\int_{\R^3}K(x)\phi_u|u|^5dx-\frac{\la}{q}\int_{\R^3}f(x)|u|^qdx.
\end{equation}
We mention here that the idea of this reduction method was proposed by Benci and Fortunato \cite{Benci1} and it is a basic strategy for studying Schr\"{o}dinger-Poisson system today.

For simplicity, the conditions in Theorem \ref{maintheorem} always hold true thought this paper and we don't assume them any longer unless specially needed.
 To know more about the solution $\phi$ of the Poisson equation in \eqref{mainequation} which can leads to a critical nonlocal term, we have the following key lemma:

\begin{lemma}\label{phi}
For every $u\in H^1(\R^3)$, we have the following conclusions:
\begin{enumerate}[$(1)$]
  \item $\phi_u(x)\geq 0$ for every $x\in\R^3$;
  \end{enumerate}
  \begin{enumerate}[$(2)$]
  \item $\|\phi_u\|_{D^{1,2}(\R^3)}^2=\int_{\R^3}K(x)\phi_u|u|^5dx$;
  \end{enumerate}
  \begin{enumerate}[$(3)$]
  \item for any $t>0$, $\phi_{tu}=t^5\phi_{u}$;
  \end{enumerate}
  \begin{enumerate}[$(4)$]
  \item if $u_n\rightharpoonup u$ in $H^1(\R^3)$, then $\phi_{u_n}\rightharpoonup \phi_u$ in $D^{1,2}(\R^3)$.
\end{enumerate}
\end{lemma}

\begin{proof}
As a direct consequence of \eqref{Lax} and \eqref{Riesz}, one can derive $(1)$, $(2)$ and $(3)$ at once.

For any $v\in D^{1,2}(\R^3)$, then $v\in L^6(\R^3)$ by \eqref{Sobolev} and hence $K(x)v\in L^6(\R^3)$ because $K(x)$ is bounded. Since $u_n\rightharpoonup u$ in $H^1(\R^3)$, then $u_n\rightharpoonup u$ in $L^6(\R^3)$ and $|u_n|^5\rightharpoonup |u|^5$ in $L^\frac{6}{5}(\R^3)$. Therefore
\[
(\phi_{u_n},v)_{D^{1,2}(\R^3)}=\int_{\R^3}K(x)|u_n|^5vdx\to\int_{\R^3}K(x)|u_n|^5vdx=(\phi_{u},v)_{D^{1,2}(\R^3)}
\]
which implies that $(4)$ is true.
\end{proof}

Furthermore, by $(2)$ of Lemma \ref{phi}, H\"{o}lder's inequality and \eqref{Sobolev}, one has
\[
\|\phi_{u}\|^2_{D^{1,2}}=\int_{\R^3}K(x)\phi_{u}|u|^5dx\leq |K|_\infty|\phi_{u}|_6|u|^5_6\leq |K|_\infty S^{-\frac{1}{2}}\|\phi_{u}\|_{D^{1,2}}|u|^5_6
\]
which implies that
\begin{equation}\label{Sobolev0}
  \|\phi_{u}\|_{D^{1,2}}\leq |K|_\infty S^{-\frac{1}{2}}|u|^5_6
\end{equation}
 and
\begin{equation}\label{Sobolev2}
  \int_{\R^3}K(x)\phi_{u}|u|^5dx\leq |K|^2_\infty S^{-6}\|u\|^{10}.
\end{equation}

Then from $f(x)\in L^{\frac{2}{2-q}}(\R^3)$ we have that the functional $J$ given by \eqref{functional} is well-defined on $H^1(\R^3)$ and is of $C^1(H^1(\R^3),\R)$ class (see \cite{Willem}), and for any $v\in H^1(\R^3)$ one has
\[
\langle J^{\prime}(u),v\rangle=\int_{\R^3}\nabla u\nabla v+uvdx-\int_{\R^3}K(x)\phi_u|u|^3uvdx-\lambda\int_{\R^3}f(x)|u|^{q-2}uvdx.
\]
It is standard to verify that a critical point $u\in H^1(\R^3)$ of the functional $J$ corresponds to a
weak solution $(u, \phi_u)\in H^1(\R^3)\times D^{1,2}(\R^3)$ of \eqref{mainequation}. In other words, if we can seek a critical point of the functional $J$, then the system \eqref{mainequation} is solvable. In the following, we call $(u, \phi_u)$ is a positive solution of \eqref{mainequation} if $u$ is a positive critical of the functional $J$. And $(u, \phi_u)$ is a least energy solution of \eqref{mainequation} if the critical point $u$ of the functional $J$ verifies
\[
J(u)=\min_{v\in \mathcal{S}}J(v),
\]
where $\mathcal{S}:=\big\{u\in H^1(\R^3)\backslash\{0\}:J^{\prime}(u)=0\big\}$.

Motivated by the well-known Br\'{e}zis-Lieb lemma \cite{Brezis2}, we have the following important lemma to prove the convergence of Schr\"{o}dinger-Poisson system \eqref{mainequation} involving a critical nonlocal term.

\begin{lemma}\label{weaklemma}
Let $r \geq 1$ and $\Omega$ be an open subset of $\R^N$. Suppose that $u_n\rightharpoonup u$
in $L^r(\Omega)$, and $u_n\to u$ $a.e.$ in $\Omega$ as $n\to\infty$, then
\[
|u_n|^p-|u_n-u|^p-|u|^p\to 0\ \ \text{in}\ \ L^{\frac{r}{p}}(\Omega)
\]
as $n\to\infty$ for any $p\in[1, r]$.
\end{lemma}

\begin{proof}
The proof is standard, so we omit it and the reader can refer in \cite[Lemma 2.2]{Li} for the detail proof.
\end{proof}

\begin{lemma}
If $u_n\rightharpoonup u$ in $H^1(\R^3)$, then going to a subsequence if necessary, we derive
\begin{equation}\label{weak1}
  |u_n|^5-|u_n-u|^5-|u|^5\to 0\ \ \text{in}\ \ L^{\frac{6}{5}}(\R^3),
\end{equation}
\begin{equation}\label{weak2}
  \phi_{u_n}-\phi_{u_n-u}-\phi_{u}\to 0\ \ \text{in}\ \ D^{1,2}(\R^3),
 \end{equation}
\begin{equation}\label{weak3}
\int_{\R^3}K(x)\phi_{u_n}|u_n|^5dx-\int_{\R^3}K(x)\phi_{u_n-u}|u_n-u|^5dx-\int_{\R^3}K(x)\phi_{u}|u|^5dx\to 0,
\end{equation}
and
\begin{equation}\label{weak4}
  \int_{\R^3}K(x)\phi_{u_n}|u_n|^3u_n\varphi dx-\int_{\R^3}K(x)\phi_{u}|u|^3u\varphi dx\to 0
\end{equation}
for any $\varphi \in C^\infty_0(\R^3)$.
\end{lemma}

\begin{proof}
Since $u_n\rightharpoonup u$ in $H^1(\R^3)$, then $u_n\rightharpoonup u$ in $L^6(\R^3)$. And $u_n\to u$ $a.e.$ in $\R^3$ because $u_n\to u$ in $L_{\text{loc}}^s(\R^3)$ with $1\leq s<6$ in the sense of a subsequence.
If we take $r=6$ and $p=5$ in Lemma \ref{weaklemma}, one has \eqref{weak1} immediately.

It follows from $(2)$ of Lemma \eqref{phi} and H\"{o}lder's inequality that
\begin{eqnarray*}
 \big|(\phi_{u_n}-\phi_{u_n-u}-\phi_{u},w)_{D^{1,2}(\R^3)}\big| &=& \big| \int_{\R^3}\nabla(\phi_{u_n}-\phi_{u_n-u}-\phi_{u})\nabla wdx \big|\\
   &=&  \big|\int_{\R^3}K(x)(|u_n|^5-|u_n-u|^5-|u|^5) wdx \big|\\
    &\leq& |K|_\infty|w|_6 \big|(|u_n|^5-|u_n-u|^5-|u|^5)\big|_{\frac{6}{5}},
\end{eqnarray*}
which implies that
\begin{eqnarray*}
 && \sup_{w\in D^{1,2}(\R^3),\ \ \|w\|_{D^{1,2}(\R^3)}=1}\big|(\phi_{u_n}-\phi_{u_n-u}-\phi_{u},w)_{D^{1,2}(\R^3)}\big| \\
   &\leq&|K|_\infty \big|(|u_n|^5-|u_n-u|^5-|u|^5)\big|_{\frac{6}{5}}\stackrel{\mathrm{\eqref{weak1}}}{\to}0,
  \end{eqnarray*}
hence \eqref{weak2} holds.

Using \eqref{weak2}, one has $\phi_{u_n}-\phi_{u_n-u}-\phi_{u}\to 0$ in $L^6(\R^3)$. Since $\{u_n\}$ is bounded in $L^6(\R^3)$, then by H\"{o}lder's inequality,
\begin{eqnarray*}
|A_1| &:=& \big|\int_{\R^3}K(x)\big(\phi_{u_n}-\phi_{u_n-u}-\phi_{u}\big)|u_n|^5dx\big| \\
   &\leq&|K|_\infty |u_n|_6^5 \big|\phi_{u_n}-\phi_{u_n-u}-\phi_{u}\big|_6\to 0.
\end{eqnarray*}
Similarly, one can deduce that
\[
A_2:=\int_{\R^3}K(x)\big(\phi_{u_n}-\phi_{u_n-u}-\phi_{u}\big)|u|^5dx\to 0.
\]
In view of \eqref{Sobolev0}, $\{\phi_{u_n-u}\}$ is bounded in $L^6(\R^3)$, then using H\"{o}lder's inequality again,
 \begin{eqnarray*}
% \nonumber to remove numbering (before each equation)
 |A_3|  &:=&\big|\int_{\R^3}K(x)\phi_{u_n-u}\big(|u_n|^5-|u_n-u|^5-|u|^5\big)dx\big|  \\
   &\leq&|K|_\infty |\phi_{u_n-u}|_6 \big|(|u_n|^5-|u_n-u|^5-|u|^5)\big|_{\frac{6}{5}} \stackrel{\mathrm{\eqref{weak1}}}{\to}0.
\end{eqnarray*}
As $u_n\rightharpoonup u$ in $H^1(\R^3)$, then one has $\phi_{u_n}\rightharpoonup \phi_{u}$ in $D^{1,2}(\R^3)$ by $(4)$ of Lemma \ref{phi} and thus $\phi_{u_n}\rightharpoonup \phi_{u}$ in $L^{6}(\R^3)$. Clearly, $K(x)|u|^5\in L^{\frac{6}{5}}(\R^3)$, thus
\[
A_4:=\int_{\R^3}K(x)(\phi_{u_n}-\phi_{u})|u|^5dx\to 0.
\]
By $u_n\rightharpoonup u$ in $H^1(\R^3)$, one has $|u_n|^5\rightharpoonup |u|^5$ in $L^\frac{6}{5}(\R^3)$. Since $K(x)\phi_u\in L^{6}(\R^3)$, then
\[
A_5:=\int_{\R^3}K(x)\phi_{u}(|u_n|^5-|u|^5)dx\to 0.
\]
Consequently,
\begin{eqnarray*}
 && \int_{\R^3}K(x)\phi_{u_n}|u_n|^5dx-\int_{\R^3}K(x)\phi_{u_n-u}|u_n-u|^5dx-\int_{\R^3}K(x)\phi_{u}|u|^5dx  \\
    &=& A_1-A_2+A_3+ A_4+A_5\\
    &\to& 0,
\end{eqnarray*}
which shows that \eqref{weak3} is true.

Since $u_n\rightharpoonup u$ in $H^1(\R^3)$, then one can deduce again that $\phi_{u_n}\rightharpoonup \phi_u$ in $L^6(\R^3)$. By $K(x)|u|^3u\varphi\in L^{\frac{6}{5}}(\R^3)$ because $\varphi\in C^\infty_0(\R^3)$, one has
\[
 \int_{\R^3}K(x)\phi_{u_n}|u|^3u \varphi dx-\int_{\R^3}K(x)\phi_{u}|u|^3u\varphi dx\to 0.
\]
On the other hand, by means of H\"{o}lder's inequality and $\{\phi_{u_n}\}$ is bounded in $L^6(\R^3)$,
\begin{eqnarray*}
&&\bigg|\int_{\R^3}K(x)\phi_{u_n}|u_n|^3u_n\varphi dx-\int_{\R^3}K(x)\phi_{u_n}|u|^3u\varphi dx\bigg|\\
   &\leq &|K|_\infty   \int_{\supp \varphi}|\phi_{u_n}||\varphi|\big||u_n|^3u_n-|u|^3u\big|dx                                                             \\
   &\leq &|K|_\infty |\phi_{u_n}|_{6}|\varphi|_{\infty}\bigg(\int_{\supp \varphi}\big||u_n|^3u_n-|u|^3u\big|^{\frac{6}{5}}dx\bigg)^{\frac{5}{6}}\to 0,
  \end{eqnarray*}
where we have used $u_n\to u$ in $L_{\text{loc}}^s(\R^3)$ with $1\leq s<6$ in the sense of a subsequence.
As a consequence of the above two facts, one has
\[
\int_{\R^3}K(x) \phi_{u_n}|u_n|^3 u_n\varphi dx-\int_{\R^3}K(x)\phi_{u}|u|^3 u\varphi dx
   \to 0.
\]
The proof is complete.
\end{proof}

\begin{lemma}\label{Mountainpass}
There exists $\lambda_0=\lambda_0(q,S,f,K)>0$ such that
 the functional $J(u)$ satisfies the Mountain-pass geometry around $0\in H^1(\R^3)$ for any $\lambda\in (0, \lambda_0)$, that is,
\begin{enumerate}[$(i)$]
  \item there exist $\alpha,\rho>0$ such that $J(u)\geq \alpha>0$ when $\|u\|=\rho$ and $\lambda\in (0, \lambda_0)$;
      \item there exists $e\in H^1(\R^3)$ with $\|e\|>\rho$ such that $J(e)<0$.
\end{enumerate}
\end{lemma}

 \begin{proof}
$(i)$ It follows from \eqref{Sobolev2} and H\"{o}lder's inequality that
\begin{eqnarray*}
J(u) &=& \frac{1}{2}\|u\|^2-\frac{1}{10}\int_{\R^3}K(x)\phi_u|u|^5dx-\frac{\la}{q}\int_{\R^3}f(x)|u|^qdx \\
   &\geq& \frac{1}{2}\|u\|^2-\frac{|K|_\infty^2}{10S^6}\|u\|^{10}-\frac{\la}{q}|f|_{\frac{2}{2-q}}\|u\|^q \\
   &=& \|u\|^q\bigg(\frac{1}{2}\|u\|^{2-q}-\frac{|K|_\infty^2}{10S^6}\|u\|^{10-q}-\frac{\la}{q}|f|_{\frac{2}{2-q}}\bigg) \\
   &\geq&\bigg(\frac{5S^6(2-q)}{|K|_\infty^2}\bigg)^{\frac{q}{8}} \bigg[\frac{4}{10-q}\bigg(\frac{5S^6(2-q)}{|K|_\infty^2(10-q)}\bigg)^{\frac{2-q}{8}}-\frac{\la}{q}|f|_{\frac{2}{2-q}}\bigg],
\end{eqnarray*}
Therefore if we set
\begin{equation}\label{la0}
  \rho=\bigg(\frac{5S^6(2-q)}{|K|_\infty^2}\bigg)^{\frac{1}{8}} >0\ \  \text{and}\ \
\lambda_0=\frac{4q}{(10-q)|f|_{\frac{2}{2-q}}}\bigg(\frac{5S^6(2-q)}{|K|_\infty^2(10-q)}\bigg)^{\frac{2-q}{8}}>0,
\end{equation}
 then there exists $\alpha>0$ such that $J(u)\geq \alpha>0$ when $\|u\|=\rho>0$ for any $\lambda\in (0,\lambda_0)$.

$(ii)$ Choosing $u_0\in H^1(\R^3)\backslash\{0\}$, then since $f(x)$ is nonnegative, one has
\begin{eqnarray*}
J(tu_0) &=& \frac{t^2}{2}\|u_0\|^2-\frac{t^{10}}{10}\int_{\R^3}K(x)\phi_{u_0}|u_0|^5dx-\frac{\la t^q}{q}\int_{\R^3}f(x)|u_0|^qdx \\
   &\leq& \frac{t^2}{2}\|u_0\|^2-\frac{t^{10}}{10}\int_{\R^3}K(x)\phi_{u_0}|u_0|^5dx\to -\infty
    \end{eqnarray*}
as $t\to+\infty$. Hence letting $e=t_0u_0\in H^1(\R^3)\backslash\{0\}$ with $t_0$ sufficiently large, one has $\|e\|>\rho$ and $J(e)<0$.
\end{proof}

By Lemma \ref{Mountainpass}, a $(PS)$ sequence of the functional $\Phi(u)$ at the level
\begin{equation}\label{Mountainpass1}
 c:=\inf_{\gamma\in \Gamma}\max_{t\in [0,1]}J(\gamma(t))>0,
\end{equation}
can be constructed, where the set of paths is defined as
\begin{equation}\label{Mountainpass2}
  \Gamma:=\big\{\gamma\in C([0,1],H^1(\R^3)):\gamma(0)=0, J(\gamma(1))<0\big\}.
\end{equation}
In other words, there exists a sequence $\{u_n\}\subset H^1(\R^3)$ such that
\begin{equation}\label{Mountainpass3}
  J(u_n)\to c,\ \ J^{\prime}(u_n)\to 0\ \  \text{as}\ \ n\to \infty.
\end{equation}

\begin{remark}
It is easy to see that
\begin{equation}\label{compare}
  c\leq \inf_{u\in H^1(\R^3)\backslash\{0\}}\max_{t\geq 0}J(tu).
\end{equation}
Indeed, for any $u\in H^1(\R^3)\backslash\{0\}$, similar to Lemma \ref{Mountainpass} $(ii)$ there exists a sufficiently large $t_0>0$ such that $J(t_0u)<0$. Let us choose $\gamma_0(t)=tt_0u$, therefore $\gamma_0\in C([0,1],H^1(\R^3))$ and moreover $\gamma_0\in \Gamma$, thus
\[
c\leq \max_{t\in [0,1]}J(\gamma_0(t))=\max_{t\in [0,1]}J(tt_0u)=\max_{t\in [0,t_0]}J(tu)\leq \max_{t\geq 0}J(tu).
\]
Since $u\in H^1(\R^3)\backslash\{0\}$ in arbitrary, then \eqref{compare} holds.
\end{remark}

Because of the appearance of the critical nonlocal term, we have to estimate the Mountain-pass value given by \eqref{Mountainpass1} carefully.
To do it, we choose the extremal function
$$
U_{\epsilon,x_0}(x)=\frac{(3\epsilon^2)^{\frac{1}{4}}}{(\epsilon^2+|x-x_0|^2)^{\frac{1}{2}}}
$$
to solve $-\Delta u=u^5$ in $\R^3$, where $x_0$ is given in condition $(K)$. Let
$\varphi\in C_0^{\infty}(\R^3)$ be a cut-off function verifying that $0\leq \varphi(x)\leq 1$ for all $x\in \R^3$, $\supp \varphi\subset B_{2}(x_0)$, and $\varphi(x)\equiv 1$ on $B_1(x_0)$.
Set $v_{\epsilon,x_0}=\varphi U_{\epsilon,x_0}$,
then thanks to the asymptotic estimates from \cite{Brezis1}, we have
\begin{equation}\label{estimate1}
  |\nabla v_{\epsilon,x_0}|_2^2=S^{\frac{3}{2}}+O(\epsilon), \ \  |v_{\epsilon,x_0}|_6^2=S^{\frac{1}{2}}+O(\epsilon)
\end{equation}
and
\begin{equation}\label{estimate2}
  |v_{\epsilon,x_0}|_2^2=O(\epsilon).
\end{equation}

\begin{lemma}\label{estimate}
Assume $1<q<2$, then the the Mountain-pass value given by \eqref{Mountainpass1}
\[
   c<\frac{2}{5}|K|_\infty^{-\frac{1}{2}}S^{\frac{3}{2}}-C_0\la^{\frac{2}{2-q}},\ \ \text{where}\ \ C_0=\frac{2(2-q)}{5q}\bigg(\frac{(10-q)|f|_{\frac{2}{2-q}}}{8}\bigg)^{\frac{2}{2-q}}>0
\]
for any $\lambda\in(0,\lambda_0)$ and $S$ is the best Sobolev constant given in \eqref{Sobolev}.
\end{lemma}

\begin{proof}
Firstly, we claim that there exist $t_1,t_2\in(0,+\infty)$ independent of $\epsilon,\lambda$ such that $\max_{t\geq 0}J(tv_{\epsilon,x_0})=J(t_\epsilon v_{\epsilon,x_0})$ and
\begin{equation}\label{2.1g}
  0<t_1<t_\epsilon<t_2<+\infty.
\end{equation}
Indeed, by the fact that $\lim_{t\to +\infty}J(tv_{\epsilon,x_0})=-\infty$ and $(i)$ of Lemma \ref{Mountainpass}, there exists $t_\epsilon>0$ such that
$$
\max\limits_{t\geq 0}J(tv_{\epsilon,x_0})=J(t_\epsilon v_{\epsilon,x_0}),\ \  \frac{d}{dt}J(tv_{\epsilon,x_0})=0, \ \ \frac{d^2}{dt^2}J(tv_{\epsilon,x_0})<0
$$
which imply that
\begin{equation}\label{2.1h}
  \|v_{\epsilon,x_0}\|^2-\lambda t_\epsilon^{q-2}\int_{\R^3}f(x)|v_{\epsilon,x_0}|^qdx=t_\epsilon^8\int_{\R^3}K(x)\phi_{v_{\epsilon,x_0}}|v_{\epsilon,x_0}|^5dx
\end{equation}
and
\begin{equation}\label{2.1i}
  \|v_{\epsilon,x_0}\|^2-9t_\epsilon^8\int_{\R^3}K(x)\phi_{v_{\epsilon,x_0}}|v_{\epsilon,x_0}|^5dx-\lambda (q-1) t_\epsilon^{q-2}\int_{\R^3}f(x)|v_{\epsilon,x_0}|^qdx<0.
\end{equation}
It follows from \eqref{2.1h} that $t_\epsilon$ is bounded from above. On the other hand,
combing with \eqref{2.1h} and \eqref{2.1i}, one has
 \[
 -8\int_{\R^3}K(x)\phi_{v_{\epsilon,x_0}}|v_{\epsilon,x_0}|^5dx<\la(q-2)t_\epsilon^{q-10}\int_{\R^3}f(x)|v_{\epsilon,x_0}|^qdx
 \]
which yields that $t_\epsilon$ is bounded from below because $q\in(1,2)$. Hence \eqref{2.1g} is true.

Let us define
 \begin{eqnarray*}
g(t) &:=& \frac{t^2}{2}\int_{\R^3}|\nabla v_{\epsilon,x_0}|^2dx-\frac{ t^{10}}{10}\int_{\R^3}K(x)\phi_{v_{\epsilon,x_0}}|v_{\epsilon,x_0}|^5dx \\
   &:=& C_1t^2-C_2t^{10},
\end{eqnarray*}
where
$$
C_1=\frac{1}{2}\int_{\R^3}|\nabla v_{\epsilon,x_0}|^2dx,\ \ C_2=\frac{1}{10}\int_{\R^3}K(x)\phi_{v_{\epsilon,x_0}}|v_{\epsilon,x_0}|^5dx.
$$
By some elementary calculations, we have
\begin{equation}\label{ggggg}
  \max_{t\geq 0}g(t) = \frac{4(C_1)^{\frac{5}{4}}}{5(5C_2)^{{\frac{1}{4}}}}
   = \frac{2}{5}\frac{   \Big( \int_{\R^3}|\nabla v_{\epsilon,x_0}|^2dx \Big)^{  \frac{5}{4}    }    }      { \Big(  \int_{\R^3}K(x)\phi_{v_{\epsilon,x_0}}|v_{\epsilon,x_0}|^5dx  \Big)^{  \frac{1}{4}    }       }.
\end{equation}
In order to further estimate the formula \eqref{ggggg}, we first get the following estimate:
\begin{eqnarray}\label{KKKKK}
\nonumber && \int_{\R^3}\big[K(x_0)-K(x)\big]|v_{\epsilon,x_0}|^6dx \\
\nonumber &\stackrel{\mathrm{(K)}}{\leq}& \nonumber 3^{\frac{1}{4}}C\int_{|x-x_0|<\delta}\frac{|x-x_0|^\beta\epsilon^3}{(\epsilon^2+|x-x_0|^2)^{3}}dx+3^{\frac{1}{4}}2|K|_\infty \int_{|x-x_0|\geq\delta}\frac{\epsilon^3}{(\epsilon^2+|x-x_0|^2)^{3}}dx \\
\nonumber&\leq& \nonumber C\epsilon^3\int_{0}^{\delta}\frac{r^{2+\beta}}{(\epsilon^2+r^2)^3}dr+C\epsilon^3\int_{\delta}^{+\infty}r^{-4}dr \\
\nonumber &\leq& C \epsilon^\beta\int_{0}^{+\infty}\frac{r^{2+\beta}}{(1+r^2)^3}dr+C\delta^{-3}\epsilon^3\\
&\leq& C\epsilon^\beta+C\epsilon^3\leq C\epsilon^\beta,
\end{eqnarray}
where we use the fact that $\beta\in[1,3)$ in the last two inequalities.
Next the Poisson equation $-\Delta \phi_{v_\epsilon,x_0}=K(x)|v_{\epsilon,x_0}|^5$ and Cauchy's inequality give
\begin{align*}
 \int_{\R^3} K(x)|v_{\epsilon,x_0}|^6dx &=  \int_{\R^3}\nabla \phi_{v_\epsilon,x_0} \nabla |v_{\epsilon,x_0}| dx\\
   &\leq \frac{1}{2|K|_\infty} \int_{\R^3}|\nabla \phi_{v_\epsilon,x_0}|^2dx+ \frac{|K|_\infty}{2} \int_{\R^3}|\nabla {v_{\epsilon,x_0}}|^2dx\\
   &=\frac{1}{2|K|_\infty} \int_{\R^3}K(x)\phi_{v_{\epsilon,x_0}}|{v_{\epsilon,x_0}}|^5dx+ \frac{|K|_\infty}{2} \int_{\R^3}|\nabla {v_{\epsilon,x_0}}|^2dx
\end{align*}
which implies that
\begin{eqnarray*}
\int_{\R^3}K(x)\phi_{v_{\epsilon,x_0}}|{v_{\epsilon,x_0}}|^5dx &\geq& 2|K|_\infty\int_{\R^3} K(x)|v_{\epsilon,x_0}|^6dx- |K|_\infty^2\int_{\R^3}|\nabla {v_{\epsilon,x_0}}|^2dx \\
   &\stackrel{\mathrm{\eqref{KKKKK}}}{\geq}&2|K|_\infty^2\int_{\R^3}|v_{\epsilon,x_0}|^6dx-C\epsilon^\beta- |K|_\infty^2\int_{\R^3}|\nabla {v_{\epsilon,x_0}}|^2dx  \\
   &\stackrel{\mathrm{\eqref{estimate1}}}{=}& |K|_\infty^2S^{\frac{3}{2}}+O(\epsilon).\ \ \ \    \Big(\beta\in[1,3)\Big)
\end{eqnarray*}
As a consequence of the above fact, one has
\begin{equation}\label{gggggg}
 \max_{t\geq 0}g(t)\stackrel{\mathrm{\eqref{ggggg}}}{\leq} \frac{2}{5}\frac   { \Big(S^{\frac{3}{2}}+O(\epsilon)\Big)^{   \frac{5}{4}  }   }
{  \Big(|K|_\infty^2S^{\frac{3}{2}}+O(\epsilon)\Big)^{   \frac{1}{4}  }   }
=\frac{2}{5}|K|_\infty^{-\frac{1}{2}}S^{\frac{3}{2}}+O(\epsilon).
\end{equation}
On the other hand, for $\epsilon>0$ with $\epsilon<1$ we have
 \begin{eqnarray}\label{ggggggg}
\nonumber \lambda t_\epsilon^q\int_{\R^3}f(x)|v_{\epsilon,x_0}|^qdx &=& \lambda t_\epsilon^q\int_{B_2(x_0)}f(x)|v_{\epsilon,x_0}|^qdx \\
\nonumber    & \stackrel{\mathrm{\eqref{2.1g}}}{\geq } & C\lambda\int_{B_{1}(x_0)}f(x)\frac{\epsilon^{\frac{q}{2}}}{(\epsilon^2+|x-x_0|^2)^{\frac{q}{2}}}dx  \\
    &\geq &  C\lambda\big(\frac{\epsilon}{2}\big)^{\frac{q}{2}}\int_{B_{1}(x_0)}f(x)dx:= C_3\la\epsilon^{\frac{q}{2}},
\end{eqnarray}
where $C_3\in (0,+\infty)$ is a constant since $f(x)\in L^{\frac{2}{2-q}}(\R^3)$ and then $f(x)\in L^1_{\text{loc}}(\R^3)$.

We have proved $\max_{t\geq 0}J(tv_\epsilon)=J(t_\epsilon v_\epsilon)$ at the beginning, that is,
\begin{eqnarray}\label{zzzzz}
\nonumber \max_{t\geq 0}J(tv_{\epsilon,x_0})&=& \frac{t_\epsilon^2}{2}\int_{\R^3}|\nabla v_{\epsilon,x_0}|^2+|v_{\epsilon,x_0}|^2dx
-\frac{t_\epsilon^{10}}{10}\int_{\R^3}K(x)\phi_{v_{\epsilon,x_0}}|v_{\epsilon,x_0}|^5dx\\
\nonumber &&-\frac{\lambda t_\epsilon^q}{q}\int_{\R^3}f(x)|v_{\epsilon,x_0}|^qdx \\
\nonumber    &=&g(t_\epsilon)+\frac{t_\epsilon^2}{2}|v_{\epsilon,x_0}|_2^2-\frac{\lambda t_\epsilon^q}{q}\int_{\R^3}f(x)|v_{\epsilon,x_0}|^qdx  \\
   &\leq&  \frac{2}{5}|K|_\infty^{-\frac{1}{2}}S^{\frac{3}{2}}+CO(\epsilon)-C_3\la \epsilon^{\frac{q}{2}},
\end{eqnarray}
where we have used \eqref{estimate2}, \eqref{2.1g}, \eqref{gggggg} and \eqref{ggggggg} in the last inequality.

Since $1<q<2$, then there exists sufficiently small $\epsilon>0$ such that
\[
CO(\epsilon)-C_3\la\epsilon^{\frac{q}{2}}<-C_0\la^{\frac{2}{2-q}},
\]
which indicates that $c<\frac{2}{5}|K|_\infty^{-\frac{1}{2}}S^{\frac{3}{2}}-C_0\la^{\frac{2}{2-q}}$ by \eqref{compare} and \eqref{zzzzz}.
\end{proof}

 \begin{lemma}\label{Willemlemma}
(see \cite[Theorem A.2]{Willem}) Let $\Omega$ be an open subset of $\R^3$ and assume that $|\Omega|<+\infty$, $1\leq p,r<+\infty$, $g\in C(\overline{\Omega}\times\R)$ and
\[
|g(x,u)|\leq c(1+|u|^{\frac{p}{r}}).
\]
Then, for every $u\in L^p(\Omega)$, $g(\cdot,u)\in L^r(\Omega)$ and the operator $A:L^p(\Omega)\to L^r(\Omega)$ defined by
\[
Au=g(x,u)
\]
is continuous.
\end{lemma}

\begin{lemma}\label{nonlinearity}
Assume $f\in L^{\frac{2}{2-q}}(\R^3)$ and $u_n\rightharpoonup u$ in $H^1(\R^3)$, then going to a subsequence if necessary, one has
\begin{equation}\label{nonlinearity1}
  \int_{\R^3}f(x)|u_n|^qdx\to \int_{\R^3}f(x)|u|^qdx
\end{equation}
and
\begin{equation}\label{nonlinearity2}
\int_{\R^3}f(x)|u_n|^{q-2}u_n\varphi dx\to \int_{\R^3}f(x)|u|^{q-2}u\varphi dx
\end{equation}
for any $\varphi\in C_0^\infty(\R^3)$.
\end{lemma}

\begin{proof}
Since $u_n\rightharpoonup u$ in $H^1(\R^3)$, then $u_n\to u$ in $L_{\text{loc}}^s(\R^3)$ with $1\leq s<6$ and $u_n\to u$ $a.e.$ in $\R^3$ in the sense of a subsequence. Since $f\in L^{\frac{2}{2-q}}(\R^3)$, for any $\epsilon>0$ there exists $R=R(\epsilon)>0$ such that
\[
\int_{B_R^c}|f(x)|^{\frac{2}{2-q}}dx\leq \epsilon.
\]
As $\{u_n\}$ is uniformly bounded in $H^1(\R^3)$, $\{u_n\}$ and $u$ are uniformly bounded in $L^2(\R^3)$. Therefore by using H\"{o}lder's inequality and Minkowski's inequality, one has
\begin{eqnarray*}
 \bigg|\int_{B_R^c}f(x)|u_n|^qdx-\int_{B_R^c}f(x)|u|^qdx\bigg|
     &\leq&\bigg(\int_{B_R^c}|f(x)|^{\frac{2}{2-q}}dx\bigg)^{\frac{2-q}{2}}
     \big||u_n|^q-|u|^q\big|_{\frac{2}{q}}\\
    &\leq&  \bigg(\int_{B_R^c}|f(x)|^{\frac{2}{2-q}}dx\bigg)^{\frac{2-q}{2}}\big(|u_n|_2^q+|u|_2^q\big)\leq C\epsilon.
 \end{eqnarray*}
Let $g(x,u)=|u|^q$, then $p:=2$ and $r:=\frac{2}{q}>1$ as in Lemma \ref{Willemlemma}. Since $u_n\to u$ in $L^2(B_R)$, then $g(x,u_n)\to g(x,u)$ in $L^{\frac{2}{q}}(B_R)$ by Lemma \ref{Willemlemma}. Thus
\begin{eqnarray*}
 \bigg|\int_{B_R}f(x)|u_n|^qdx-\int_{B_R}f(x)|u|^qdx\bigg|
     &\leq&|f|_{\frac{2}{2-q}}\bigg(\int_{B_R}\big||u_n|^q-|u|^q\big|^{\frac{2}{q}}dx\bigg)^{\frac{q}{2}}\\
    &=&  |f|_{\frac{2}{2-q}}\bigg(\int_{B_R}|g(x,u_n)-g(x,u)|^{\frac{2}{q}}dx\bigg)^{\frac{2}{q}}\to 0
 \end{eqnarray*}
which reveals \eqref{nonlinearity1} holds together the above fact. The proof of \eqref{nonlinearity2} is similar to that of \eqref{nonlinearity1}, we omit the details.
\end{proof}

\section{The proof of Theorem \ref{maintheorem}}

In this section, we will prove the Theorem \ref{maintheorem} in detail.

\subsection{Existence of a first positive solution for \eqref{mainequation}}
\vskip3mm
\begin{proof}
Let $\la_0>0$ be given as in \eqref{la0}, then for any
$\lambda\in (0,\lambda_0)$, by Lemma \ref{Mountainpass}, there exists a sequence $\{u_n\}\subset H^1(\R^3)$ verifying \eqref{Mountainpass3}.
We can show that the sequence $\{u_n\}$ is bounded in $H^1(\R^3)$. Indeed,
\begin{align*}
 c+1+o(1)\|u_n\|&\geq J(u_n)-\frac{1}{10}\langle J^{\prime}(u_n),u_n\rangle \\
       &=\frac{2}{5}\|u_n\|^2-\lambda(\frac{1}{q}-\frac{1}{10})\int_{\R^3}f(x)|u_n|^qdx        \\
         & \geq \frac{2}{5}\|u_n\|^2-\frac{10-q}{10q}\lambda|f|_{\frac{2}{2-q}}\|u_n\|^q,
  \end{align*}
hence $\{u_n\}$ is bounded in $H^1(\R^3)$ by the fact that $1<q<2$. It is therefore that there exists $u_1\in H^1(\R^3)$ such that $u_n\rightharpoonup u_1$ in $H^1(\R^3)$. To end the proof, we will split it into several steps:
\vskip0.3cm
\underline{\textbf{Step 1:}} \ \ \ \   $u_1\not\equiv 0$.
\vskip0.3cm
\noindent In fact, we will argue it indirectly and just suppose that $u_1\equiv 0$. Hence it follows from \eqref{Mountainpass3} and \eqref{nonlinearity1} that
\[
J(u_n)=\frac{1}{2}\|u_n\|^2-\frac{1}{10}\int_{\R^3}K(x)\phi_{u_n}|u_n|^5dx=c+o(1)
\]
and
\[
\langle J^{\prime}(u_n),u_n\rangle=\|u_n\|^2-\int_{\R^3}K(x)\phi_{u_n}|u_n|^5dx=o(1).
\]
Thus without loss of generality, we may assume
\[
\lim_{n\to\infty}\|u_n\|^2=\lim_{n\to\infty}\int_{\R^3}K(x)\phi_{u_n}|u_n|^5dx=l,\ \  \text{and}\ \ c=\frac{2}{5}l.
\]
On the other hand, by \eqref{Sobolev2} we can deduce that
\[
\int_{\R^3}K(x)\phi_{u_n}|u_n|^5dx\leq |K|_\infty^2S^{-6}\|u_n\|^{10}
\]
which implies that $l\leq |K|_\infty^2S^{-6}l^5$. Hence either $l=0$ or $l\geq |K|_\infty^{-\frac{1}{2}}S^{\frac{3}{2}}$. But $l=0$ yields that $c=0$ which is a contradiction to \eqref{Mountainpass1}, hence $l\geq |K|_\infty^{-\frac{1}{2}}S^{\frac{3}{2}}$. However
\[
c=\frac{2}{5}l\geq \frac{2}{5}|K|_\infty^{-\frac{1}{2}}S^{\frac{3}{2}}
\]
which also yields a contradiction to Lemma \ref{estimate}. Therefore $u_1\not\equiv 0$ holds.
\vskip0.3cm
\underline{\textbf{Step 2:}} \ \ \ \   $J^{\prime}(u_1)=0$.
\vskip0.3cm
\noindent To see this, since $C_0^{\infty}(\R^3)$ is dense in $H^1(\R^3)$, then it suffices to show
\[
\langle J^{\prime}(u_1),\varphi\rangle=0\ \ \text{for any}\ \ \varphi\in C_0^{\infty}(\R^3).
\]
Indeed, as a direct consequence of \eqref{weak4}, \eqref{Mountainpass3}, \eqref{nonlinearity2}, one has
\[
\langle J^{\prime}(u_1),\varphi\rangle=\lim_{n\to\infty}\langle J^{\prime}(u_n),\varphi\rangle=0
\]
for any $\varphi\in C_0^{\infty}(\R^3)$.
\vskip0.3cm
\underline{\textbf{Step 3:}} \ \ \ \   $J(u_1)=c>0$ \ \  and \ \  $u_1(x)>0$ \ \  in \ \  $\R^3$.
\vskip0.3cm
\noindent We first show that
\begin{equation}\label{positive0}
 J(u_1)\geq -C_0\la^{\frac{2}{2-q}},\ \ \text{where}\ \ C_0=\frac{2(2-q)}{5q}\bigg(\frac{(10-q)|f|_{\frac{2}{2-q}}}{8}\bigg)^{\frac{2}{2-q}}>0.
\end{equation}
 Indeed, by means of $J^{\prime}(u_1)=0$ and H\"{o}lder's inequality, we derive
\begin{eqnarray*}
J(u_1)&=& J(u_1)-\frac{1}{10}\langle J^{\prime}(u_1),u_1\rangle \\
   &=&\frac{2}{5}\|u_1\|^2-\lambda \frac{10-q}{10q}\int_{\R^3}f(x)|u_1|^qdx \\
  &\geq& \frac{2}{5}\|u_1\|^2-\lambda \frac{10-q}{10q}|f|_{\frac{2}{2-q}}\|u_1\|^q
  \geq -C_0\la^{\frac{2}{2-q}}.
\end{eqnarray*}
Let $v_n:=u_n-u_1$, using the Br\'{e}zis-Lieb lemma \cite{Brezis2}, $J^{\prime}(u_1)=0$, \eqref{weak3}, and \eqref{positive0}, one has
\begin{eqnarray}\label{positive1}
\nonumber o(1)&=&\langle J^{\prime}(u_n),u_n\rangle=\langle J^{\prime}(u_n),u_n\rangle-\langle J^{\prime}(u_1),u_1\rangle \\
   &=&  \|v_n\|^2- \int_{\R^3}K(x)\phi_{v_n}|v_n|^5dx +o(1)
\end{eqnarray}
and
\begin{eqnarray}\label{positive2}
\nonumber  c&=& J(u_n)-J(u_1)+J(u_1)+o(1) \\
\nonumber    &=& \frac{1}{2}\|v_n\|^2-\frac{1}{10}\int_{\R^3}K(x)\phi_{v_n}|v_n|^5dx +J(u_1)+o(1)\\
   &\geq&  \frac{1}{2}\|v_n\|^2-\frac{1}{10}\int_{\R^3}K(x)\phi_{v_n}|v_n|^5dx -C_0\la^{\frac{2}{2-q}}+o(1)
\end{eqnarray}
Just suppose that $v_n\not\to 0$ in $H^1(\R^3)$, and we may assume that $\lim_{n\to\infty}\|v_n\|^2=l_1>0$. It follows from \eqref{positive1} and
\[
\int_{\R^3}K(x)\phi_{v_n}|v_n|^5dx\stackrel{\mathrm{\eqref{Sobolev2}}}{\leq}|K|_\infty^2S^{-6}\|v_n\|^{10}
\]
that we can derive $l_1\geq |K|_\infty^{-\frac{1}{2}}S^{\frac{3}{2}}$. Hence as a consequence of \eqref{positive1} and \eqref{positive2}, one has
\begin{eqnarray*}
 c&=& \frac{1}{2}l_1-\frac{1}{10}l_1 -C_0\la^{\frac{2}{2-q}}+o(1) \\
   &\geq&\frac{2}{5} |K|_\infty^{-\frac{1}{2}}S^{\frac{3}{2}}-C_0\la^{\frac{2}{2-q}}+o(1)
  \end{eqnarray*}
which yields a contradiction to Lemma \ref{estimate}. Therefore $\|v_n\|\to 0$ in $H^1(\R^3)$ holds, or equivalently, $u_n\to u_1$ in $H^1(\R^3)$ as $n\to\infty$. Then $J(u_1)=\lim_{n\to\infty}J(u_n)=c>0$.

On the other hand, it is obvious that $|u_1|$ is also a nontivial solution of problem \eqref{mainequation} since the functional
$J$ is symmetric and invariant, hence we may assume that such a critical point does
not change sign, $i.e.$ $u_1\geq 0$. By means of the strong maximum principle and standard arguments,
see e.g. \cite{Alves,Benedetto,G. Li0,Moser,Trudinger}, we obtain that $u_1(x)> 0$ for
all $x\in\R^3$. Thus, $(u_1,\phi_{u_1})$ is a positive solution for the system \eqref{mainequation} and the proof is complete.
\end{proof}
\subsection{Existence of a second positive solution for \eqref{mainequation}}
\
\vskip3mm
Before we obtain the second positive solution, we introduce the following well-known proposition:
\begin{proposition}\label{propo}
(Ekeland's variational principle \cite{Ekeland}, Theorem 1.1)
Let $V$ be a complete metric space and $F:V\to \R\cup\{+\infty\}$ be lower semicontinuous, bounded from below. Then for any $\epsilon>0$, there exists some point $v\in V$ with
$$
F(v)\leq \inf_VF+\epsilon,\ \ F(w)\geq F(v)-\epsilon d(v,w)\ \ \text{for all}\ \  w\in V.
$$
\end{proposition}

We are in a position to show the existence of a second positive solution for \eqref{mainequation}:
\vskip3mm
\begin{proof}
The main idea of this proof comes from \cite{Shen}, we will show it for reader's convenience.
For $\rho>0$ given by Lemma \ref{Mountainpass}(i), define
$$
\overline{B}_\rho=\{u\in H^1(\R^3),\|u\|\leq \rho\},\ \ \partial B_\rho=\{u\in H^1(\R^3),\|u\|= \rho\}
$$
and clearly $\overline{B}_\rho$ is a complete metric space with the distance
$$
d(u,v)=\|u-v\|, \ \forall u,v\in \overline{B}_\rho.
$$
Lemma \ref{Mountainpass} tells us that
\begin{equation}\label{3.1a}
 J(u)|_{\partial B_\rho}\geq \alpha>0.
\end{equation}
It's obvious that the functional $J$ is lower semicontinuous and bounded from below on $\overline{B}_\rho$. We claim that
\begin{equation}\label{3.1b}
  \widetilde{c} :=\inf_{u\in \overline{B}_\rho}J(u)<0.
\end{equation}
Indeed, we chose a nonnegative function $\psi\in C_0^{\infty}(\R^3)$, and clearly $\psi\in H^1(\R^3)$. Since $1<q<2$, we have
\begin{align*}
\lim_{t\to 0}\frac{J(t\psi)}{t^q}&=\lim_{t\to 0}\frac{\frac{t^2}{2}\|\psi\|^2-\frac{t^{10}}{10}\int_{\R^3}K(x)\phi_\psi |\psi|^5 dx-\frac{\lambda t^q}{q}\int_{\R^3}f(x)|\psi|^qdx}{t^q} \\
       &=-\frac{\lambda}{q}\int_{\R^3}f(x)|\psi|^qdx<0.
\end{align*}
Therefore there exists a sufficiently small $t_0>0$ such that $\|t_0\psi\|\leq \rho$ and $J(t_0\psi)<0$, which imply that \eqref{3.1b} holds.

By Proposition \ref{propo}, for any $n\in N$ there exists $\widetilde{u}_n$ such that
\begin{equation}\label{3.1c}
\widetilde{ c}\leq J(\widetilde{u}_n)\leq \widetilde{c} +\frac{1}{n},
\end{equation}
and
\begin{equation}\label{3.1d}
  J(v)\geq J(\widetilde{u}_n)-\frac{1}{n}\|\widetilde{u}_n-v\|, \ \  \forall v\in  \overline{B}_\rho.
\end{equation}
Firstly, we claim that $\|\widetilde{u}_n\|<\rho$ for $n\in N$ sufficiently large. In fact, we will argue it by contradiction and just suppose that $\|\widetilde{u}_n\|=\rho$ for infinitely many $n$, without loss of generality, we may assume that $\|\widetilde{u}_n\|=\rho$ for any $n\in N$. It follows from \eqref{3.1a} that
$$
J(\widetilde{u}_n)\geq \alpha>0,
$$
then combing it with \eqref{3.1c}, we have $c_1\geq \alpha>0$ which is a contradiction to \eqref{3.1b}.

Next, we will show that $J^{\prime}(\widetilde{u}_n)\to 0$ in $(H^1(\R^3))^{*}$. Indeed, set
$$
v_n=\widetilde{u}_n+tu, \ \  \forall u\in B_1=\{u\in H^1(\R^3),\|u\|=1\},
$$
where $t>0$ small enough such that $2t+t^2\leq \rho^2-\|\widetilde{u}_n\|^2$ for fixed $n$ large, then
\begin{align*}
\|v_n\|^2&=\|\widetilde{u}_n\|^2+2t(\widetilde{u}_n,u)+t^2 \\
       &\leq \|\widetilde{u}_n\|^2+2t+t^2\\
         &\leq \rho^2
\end{align*}
which imply that $v_n\in \overline{B}_\rho$. So it follows from \eqref{3.1d} that
$$
J(v_n)\geq J(\widetilde{u}_n)-\frac{t}{n}\|\widetilde{u}_n-v_n\|,
$$
that is,
$$
\frac{J(\widetilde{u}_n+tu)-J(\widetilde{u}_n)}{t}\geq -\frac{1}{n}.
$$
Letting $t\to 0$, then we have $\langle J^{\prime}(\widetilde{u}_n), u\rangle\geq -\frac{1}{n}$ for any fixed $n$ large. Similarly, chose $t<0$ and $|t|$ small enough, repeating the process above we have $\langle J^{\prime}(\widetilde{u}_n), u\rangle\leq \frac{1}{n}$ for any fixed $n$ large. Therefore the conclusion
$$
\langle J^{\prime}(\widetilde{u}_n), u\rangle\to 0\ \ \text{as}\ \ n\to\infty, \ \  \forall u\in B_1
$$
implies that $J^{\prime}(\widetilde{u}_n)\to 0$ in $(H^1(\R^3))^{*}$.

Finally, we know that $\{\widetilde{u}_n\}$ is a $(PS)_{c_1}$ sequence for the functional $J$ with $c_1<0$. Since $\|\widetilde{u}_n\|< \rho$, there exists $u_2\in H^1(\R^3)$ such that $\widetilde{u}_n\rightharpoonup u_2$ in $H^1(\R^3)$.
Hence as the Step 1, Step 2 and Step 3 in Section 3.1, $J^{\prime}(u_2)=0$ and $u_2>0$. In other words, $(u_2,\phi_{u_2})$ is a positive solution for \eqref{mainequation}.
\end{proof}

\subsection{Existence of a positive least energy solution for \eqref{mainequation}}
\
\vskip3mm
To establish a positive least energy solution for problem \eqref{mainequation}, we define
$$
m:= \inf_{u\in\mathcal{S}}J(u),
$$
where $\mathcal{S}:=\big\{u\in H^1(\R^3)\backslash\{0\}:J^{\prime}(u)=0\big\}$. Firstly we have the following claims:
\vskip0.3cm
\underline{\textbf{Claim 1:}}\ \ \ \  $J(u_2)=\widetilde{c}<0$, where $u_2$ is obtained in Section 3.2.
\vskip0.3cm
\underline{\textbf{Proof of the Claim 1:}}\ \ On one hand, it follows from Fatou's lemma that $\|u_2\|\leq \mathop{\lim\inf}_{n\to\infty}\|\widetilde{u}_n\|\leq \rho$ and then $J(u_2)\geq \widetilde{c}$ by \eqref{3.1b}.

On the other hand, since $J^{\prime}(u_2)=0$, then using Fatou's lemma and \eqref{nonlinearity1} one has
\begin{align*}
 \widetilde{c}_2+o(1)&= J(\widetilde{u}_n)-\frac{1}{10}\langle J^{\prime}(\widetilde{u}_n),\widetilde{u}_n\rangle \\
       &=\frac{2}{5}\|\widetilde{u}_n \|^2-\lambda(\frac{1}{q}-\frac{1}{10})\int_{\R^3}f(x)|\widetilde{u}_n|^qdx        \\
       &=\frac{2}{5}\|\widetilde{u}_n \|^2-\lambda(\frac{1}{q}-\frac{1}{10})\int_{\R^3}f(x)|u_2|^qdx+o(1)\\
         & \geq \frac{2}{5}\|u_2\|^2-\lambda(\frac{1}{q}-\frac{1}{10})\int_{\R^3}f(x)|u_2|^qdx+o(1)\\
         &=J(u_2)-\frac{1}{10}\langle J^{\prime}(u_2),u_2\rangle+o(1)=J(u_2)+o(1).
  \end{align*}
  Thus $\widetilde{c}\geq J(u_2)$ and then $J(u_2)=\widetilde{c}<0$ by \eqref{3.1b}.
\vskip0.3cm
\underline{\textbf{Claim 2:}}\ \ \ \  $\mathcal{S}\neq\emptyset$ and $m\in(-\infty, 0)$.
\vskip0.3cm
\underline{\textbf{Proof of the Claim 2:}}\ \  It's obvious that the solutions $u_1,u_2\in \mathcal{S}$ obtained in Section 3.1 and Section 3.2, hence $\mathcal{S}\neq\emptyset$ and $m\leq \min\{J(u_1),J(u_2)\}\leq J(u_2)<0$ by Claim 1.

On the other hand, $\forall u\in \mathcal{S}$, one has
\begin{align*}
 c&= J(u)-\frac{1}{10}\langle J^{\prime}(u),u\rangle \\
       &=\frac{2}{5}\|u \|^2-\lambda(\frac{1}{q}-\frac{1}{10})\int_{\R^3}f(x)|u |^qdx        \\
         & \geq \frac{2}{5}\|u \|^2-\frac{10-q}{10q}\lambda|f|_{\frac{6}{6-q}}S^{-\frac{q}{2}}\|u \|^q,
  \end{align*}
hence $J(u)$ is coercive and bounded below on $\mathcal{S}$ by the fact that $1<q<2$, that is $m>-\infty$.
\vskip3mm
Now let us prove the existence of a least energy solution for \eqref{mainequation}:
\vskip3mm
\begin{proof}
By means of Claim 1, we can choose a minimizing sequence of $m$, that is, a sequence $\{w_n\}\subset \mathcal{S}$ satisfying
$$
J(w_n)\to m\ \  \text{as} \ \ n\to\infty\ \   \text{and}  \ \ J^{\prime}(w_n)=0.
$$
Thus $\{w_n\}$ is a $(PS)_m$ sequence of the functional $J$ with $-\infty<m<0$. It is clear that $\{w_n\}$ is bounded in $H^1(\R^3)$ and there exists $w\in H^1(\R^3)$ such that $w_n\rightharpoonup w$ in $H^1(\R^3)$. It is totally similar to Steps 1-3 in Section 3.1 that $J^{\prime}(w)=0$ and $w>0$. Hence $w\in \mathcal{S}$ and then $J(w)\geq m$. Now, we prove that $m \geq J(w)$.

In fact, since $J^{\prime}(w_n)=J^{\prime}(w)=0$, then using Fatou's lemma and \eqref{nonlinearity1} one has
\begin{align*}
 m+o(1)&= J(w_n)-\frac{1}{10}\langle J^{\prime}(w_n),w_n\rangle \\
       &=\frac{2}{5}\|w_n \|^2-\lambda(\frac{1}{q}-\frac{1}{10})\int_{\R^3}f(x)|w_n|^qdx        \\
       &=\frac{2}{5}\|w_n \|^2-\lambda(\frac{1}{q}-\frac{1}{10})\int_{\R^3}f(x)|w|^qdx+o(1)\\
         & \geq \frac{2}{5}\|w\|^2-\lambda(\frac{1}{q}-\frac{1}{10})\int_{\R^3}f(x)|w|^qdx+o(1)\\
         &=J(w)-\frac{1}{10}\langle J^{\prime}(w),w\rangle+o(1)=J(w)+o(1).
  \end{align*}
It is therefore that $J^{\prime}(w)=0$ with $J(w)=m$, and $w>0$. Consequently, $(w,\phi_w)$ is a positive least energy solution for \eqref{mainequation}.
\end{proof}

\textbf{Acknowledgements:} The authors were supported by NSFC (Grant No. 11371158), the program for
Changjiang Scholars and Innovative Research Team in University (No. IRT13066).

\end{document}